\documentclass{amsart}

\usepackage{amsmath,amssymb,amsthm}
\usepackage{mathrsfs}
\usepackage{amscd}
\usepackage[active]{srcltx}
\usepackage{verbatim}
\usepackage{enumerate}
\usepackage{graphicx}
\usepackage{bbm}

\usepackage[colorlinks,linkcolor={blue},citecolor={blue},urlcolor={red}]{
hyperref}

\allowdisplaybreaks

\let\mathcal \undefined
\def\mathcal{\mathscr}

\theoremstyle{plain}
\newtheorem{theorem}{Theorem}[section]
\theoremstyle{remark}

\newtheorem{definition}[theorem]{Definition}
\theoremstyle{plain}
\newtheorem{corollary}[theorem]{Corollary}
\newtheorem{lemma}[theorem]{Lemma}

\numberwithin{equation}{section}

\def\Z{{\mathbb Z}}

\def\R{{\mathbb R}}
\def\C{{\mathbb C}}

\renewcommand{\P}{{\mathbb P}}

\newcommand{\calC}{\mathscr{C}}
\renewcommand{\a}{\alpha}
\renewcommand{\b}{\beta}
\newcommand{\g}{\gamma}
\renewcommand{\d}{\delta}
\newcommand{\e}{\varepsilon}
\renewcommand{\l}{\lambda}

\newcommand{\B}{\mathscr{B}}

\newcommand{\n}{\Vert}
\newcommand{\one}{{{\mathbbm 1}}}

\newcommand{\dist}{{\rm dist}}
\newcommand{\diam}{{\rm diam}}
\newcommand{\supp}{{\rm supp}}

\newcommand{\een}{\end{enumerate}}
\newcommand{\bit}{\begin{itemize}}
\newcommand{\eit}{\end{itemize}}

\newcommand{\book}[1]{{\em #1}}
\newcommand{\journal}[1]{{\em #1}}
\newcommand{\volume}[1]{{\bfseries #1}}
\newcommand{\name}[1]{{\sc #1}}

\begin{document}

\author{Jan Maas}
\address{
Institute for Applied Mathematics\\
University of Bonn\\
Endenicher Allee 60\\
53115 Bonn\\
Germany} 
\email{maas@iam.uni-bonn.de}

\author{Jan van Neerven}
\address{
Delft Institute of Applied Mathematics\\
Delft University of Technology
\\P.O. Box 5031\\
2600 GA Delft\\
The Netherlands} \email{J.M.A.M.vanNeerven@tudelft.nl}

\author{Pierre Portal}
\address{Universit\'e Lille 1, Laboratoire Paul Painlev\'e, 59655 Villeneuve
d'Ascq, France}
\email{pierre.portal@math.univ-lille1.fr}

\title[Whitney coverings and the tent spaces $T^{1,q}(\gamma)$]{Whitney
coverings and the tent spaces $T^{1,q}(\gamma)$ for the Gaussian measure}

\begin{abstract}
We introduce a technique for handling Whitney decompositions in Gaussian
harmonic analysis and apply it to the study of Gaussian analogues of the classical tent
spaces $T^{1,q}$ of Coifman, Meyer and Stein.
\end{abstract}


\thanks{The first named author is supported by Rubicon subsidy 680-50-0901
of the Netherlands Organisation for Scientific Research (NWO). The second named
author is supported by VICI subsidy 639.033.604
of the Netherlands Organisation for Scientific Research (NWO)}
\date\today

\maketitle

\section{Introduction}

Much of modern harmonic analysis in euclidean spaces depends upon the fact that
the Lebesgue measure is compatible with the scalar multiplication in the sense
that for any ball $B$ in $\R^n$ we have $|2B| = 2^n |B|$; here $2B$ is the ball
with the same centre and twice the radius of $B$. Indeed, many results proved
originally in the euclidean setting have been extended to metric spaces endowed
with a doubling measure $\mu$, i.e., a measure satisfying $\mu(2B)\le C\mu(B)$
for some constant $C$ depending only upon $\mu$.

It is a simple matter to verify that the standard Gaussian measure 
$\g$ on $\R^{n}$, $$d\gamma(x) =
(2\pi)^{-n/2}\exp(-\tfrac12|x|^2)\,dx,$$
is non-doubling. In their seminal paper \cite{MM}, Mauceri and Meda found a way
around this by introducing the class of {\em admissible balls}. These are the
balls $B = B(c_B,r_B)$ in $\R^n$ satisfying the smallness condition 
$$ r_B \le \min\Big\{1,\frac1{|c_B|}\Big\}.$$
Mauceri and Meda show that admissible balls enjoy a doubling condition. Armed
with this, many results from the euclidean case can be carried over to the
Gaussian case, as long as one is able to work with admissible balls only. Mauceri
and Meda were thus able to define Gaussian counterparts of the spaces $H^1$ and
$BMO$ and extend parts of the Calder\'on-Zygmund theory to the Gaussian setting. 
Some of these results have even been extended to a more general class of locally 
doubling metric 
measure spaces in \cite{CMM1, CMM2}.

Another important tool of euclidean harmonic analysis is the Whitney covering method. 
This technique allows one to cover open sets $O$ with dyadic cubes whose sizes are 
proportional
to the distance of the cube to the complement of $O$. In the Gaussian case, one
runs into the problem that admissible cubes become very small at large distances
from the origin. As a consequence, the distance of such a cube to the exterior
of a given open set is typically much larger than the size of the cube. At first
sight, this renders Whitney covering useless as a tool in the Gaussian setting.
The purpose of this note is to show how Whitney covering, too, can be adapted to
the Gaussian setting. To illustrate its usefulness, we use it to prove an
atomic decomposition theorem and a change of aperture theorem for the Gaussian
analogue of the tent space $T^{1,q}$ of Coifman, Meyer and Stein.

\section{Admissible balls and cubes}\label{sec:adm_balls}

Throughout this paper we fix the dimension $n\ge 1$.
As usual we denote by $$B(x,r) := \{y\in\R^n: |x-y|<r\}$$ the open ball in
$\R^n$ centred at $x$ with radius $r$. Following Mauceri and Meda \cite{MM} we
begin by introducing the class of admissible balls.

\begin{definition}
For $\a>0$ we define
$$\B_{\a}:=\big\{B(x,r): \  x \in \R^{n}, \ 
0\le r\le \a m(x)\big\},$$
where
$$m(x) := \min\Big\{1,\frac{1}{|x|}\Big\}, \quad x\in\R^n.$$
The balls in $\B_{\a}$ are said to be {\em admissible at scale $\a$}.
\end{definition}

It is a fundamental observation of Mauceri and Meda \cite{MM} that admissible
balls enjoy a doubling property:

\begin{lemma}[Doubling property]\label{lem:doubling} 
Let $\a, \tau > 0.$ There exists a constant $d=d_{\alpha, \tau, n}$, depending
only upon $\alpha$, $\tau$, and the dimension $n$, such that if $B_1 =
B(c_1,r_1) \in \B_{\a}$ and $B_2 = B(c_2,r_2)$ have nonempty intersection
and $r_2 \leq \tau r_1,$ then
$$ \g(B_2)\le d \g(B_1).$$
\end{lemma}

In particular this lemma implies that for all $\a>0$ there exists a constant $d'
= d_{\a,n}'$ such that for all $B(x,r)\in \mathscr{B}_\a$ we have
$$  \g(B(x,2r))\le d' \g(B(x,r)).$$

\begin{lemma}
\label{lem:mnp1}
Let $a,b>0$ be given.
\begin{enumerate}
\item[\rm(i)] If $r\le am(x)$ and $|x-y|<br$, then $r\le a(1+ab)m(y)$.
\item[\rm(ii)] If $|x-y|<bm(x)$, then $m(x)\le (1+b)m(y)$ and $m(y)\le (2+2b)m(x)$. 
\end{enumerate}
\end{lemma}
\begin{proof}
(i): If $|y|\le 1$, then $m(y)= 1$ and $r\le am(x)\le a = am(y)$.

If $1 < |y| \leq 1+ab$, then $m(y) \geq 1/(1+ ab)$ and
$$r \le a \le a(1+a b) m(y).$$

If $|y|>1 + ab$, then $m(y) = \frac1{|y|}$ and
$$r\leq \frac{a}{|x|}
   \leq \frac{a}{|y|-br}
   \leq \frac{a}{|y|-ab}
   \leq \frac{a(1+ab)}{|y|}
      = a(1+ab) m(y).$$

(ii): Put $r'=m(x)$. Then $|x-y|< br'$ and therefore (i) (with $a=1$) implies that 
$r'\le (1+b)m(y)$. This gives the first estimate. To obtain the second 
we consider three cases. If $|x| \leq 1$, then $(2+2b)m(x) \geq 1 \geq m(y)$. 
If $|x| \geq 1$ and
$|x| \leq 2b$, then $(2+2b)m(x) \geq \frac{2+2b}{2b}\geq 1 \geq m(y)$. If $|x| \geq 1$ 
and
$|x| \geq 2b$, then $|y| \geq |x|-\frac{b}{|x|} \geq |x|-\frac{1}{2} \geq \frac{|x|}{2}$, and thus
$m(y) \leq 2m(x) \leq (2+2b)m(x)$.
\end{proof}

For $m\in \Z$ let $\Delta_m$ be the set of dyadic cubes at scale
$m$, i.e.,
$$ \Delta_m = \{2^{-m}(x+[0,1)^n) :  \ x\in \Z^n\}.$$
In the Gaussian setting the idea is to use, at every
scale, cubes whose diameter depends upon another
parameter $l \geq 0$, which keeps track of the distance from the cube
to the origin. More precisely, define the {\em layers}
$$ L_0 = [-1,1)^n, \quad L_{l} = [-2^l, 2^l)^n\setminus [-2^{l-1}, 2^{l-1})^n \
\ (l\ge 1), $$
and define, for $k\in\Z$ and $l\ge 0$,
$$ \Delta_{k,l}^\gamma =\{ Q\in \Delta_{l+k}: \ Q\subseteq L_l\},
\quad \Delta_{k}^\gamma =\bigcup_{l\ge 0}\Delta_{k,l}^\gamma,
\quad \Delta^\gamma = \bigcup_{k\ge 0} \Delta_k^\gamma.$$ Note
that $\Delta_{0,0}^\gamma$ consists of $2^n$ cubes of side length $1,$ 
and that $\Delta_{k,l}^\g = \emptyset$ for all other $k\le -2l$.
Also, if $Q\in \Delta_{k,l}^\gamma$, then $Q$ has side-length $2^{-k-l}$,
diameter $2^{-k-l}\sqrt{n}$, and its centre $x$ has norm $|x| \ge 2^{l-1}$.

Fix an integer $\kappa \geq 1.$ For each $l\ge \lceil
\frac{\kappa+1}{2}\rceil,$ the layer $L_l$ is a disjoint union of cubes in
$\Delta_{-\kappa,l}^\g$, each of which is the disjoint union of
$2^{\kappa n}$ cubes from $\Delta_{0,l}^\g$. Each such cube can be labelled by
a label $i = (i_1,\dots,i_n)\in\{1,\dots, 2^\kappa\}^n$. See Figure 1,
where
$n=2,$ $\kappa = 3,$ and the shaded cubes are the cubes from $\Delta_{0,l}^\g$ with
label $i= (5,8)$
for $l=2,3$.

\begin{figure}\label{fig:1}
\centering
\scalebox{.65}{\includegraphics{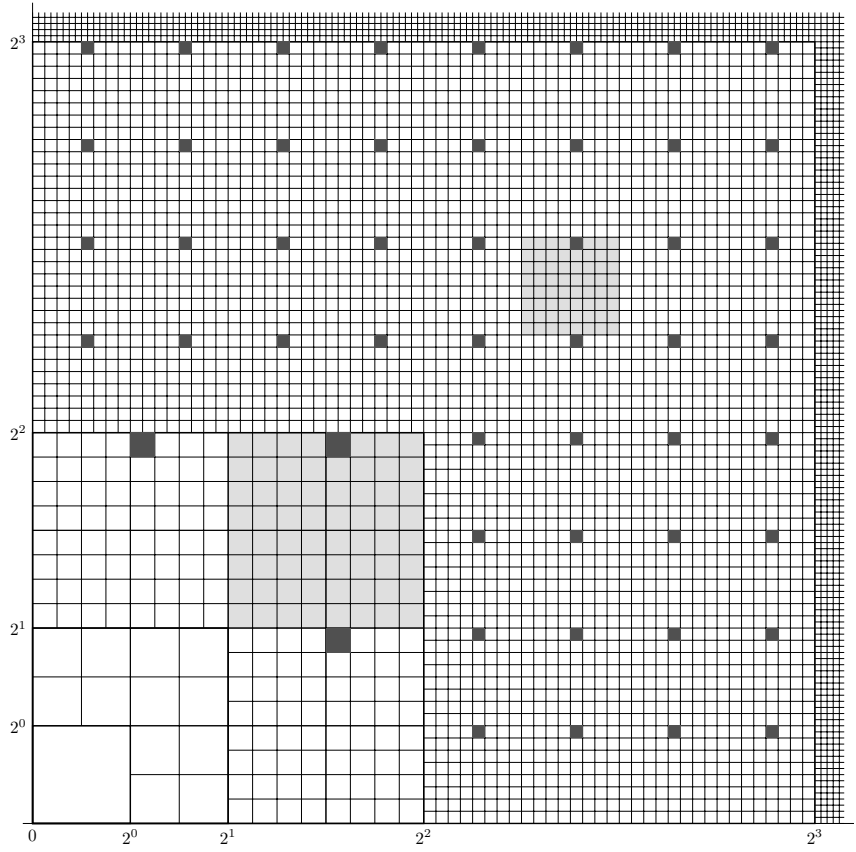}}
\caption{Subdivision of the layers $L_0, L_1, L_2,\dots$ into cubes of
$\Delta_{0,l}^\g$. The (5,8)-cubes in the layers $L_2$ and $L_3$ corresponding to the choice $\kappa=3$ are dark grey. A cube from $\Delta_{-3,2}^\g$ and a cube from $\Delta_{-3,3}^\g$ have been coloured light grey.} 
\end{figure}

For a set $A\subseteq \R^n$ we write
\begin{equation}\label{eq:AplusB}
A+\calC_\a = \{z\in\R^n: \ z \hbox{ is the centre of a ball $B\in \B_\a$ that
intersects $A$}\}.
\end{equation}

\begin{lemma}\label{lem:llmin1} 
Let $p\ge 0$ and $l\ge p+2$ 
be integers, let $Q\in \Delta_{0,l}^\g$ be given, 
and consider a ball $B=B(c_B,r_B)$ in $\B_{2^p}$
intersecting $Q$. Then we have $c_B\in L_{l-1}\cup L_l\cup L_{l+1}$. 
\end{lemma}

\begin{proof}
Suppose first that we had $c_B\in L_{l-m}$ for some $2\le m\le l$.
On the one hand, $r_B \le 2^p \le 2^{l-2}$.
On the other hand, the distance between
the layers $L_l$ and $L_{l-m}$ is at least 
$2^{l-2}+2^{l-3}+\dots + 2^{l-m} = 2^{l-1} - 2^{l-m} \ge 2^{l-2}$.
Since $B$ is open, it would follow that $B$ does not intersect 
$Q\subseteq L_l$.

The proof that $c_B$ cannot be in $L_{l+m}$ for any $m\ge 2$ is similar and requires
only cruder estimates.
\end{proof}

\begin{lemma}\label{lem:Whit}
Fix integers $p\ge 1$ and $\kappa\ge p+4$. Let $i
\in\{1,\dots,2^\kappa\}^n$ and
let $Q_1\in \Delta_{0,l_1}^\g$ and $Q_2\in \Delta_{0,l_2}^\g$ be two distinct
cubes with
the same label $i$ in the layers $L_{l_1}$ and $L_{l_2}$ with 
$l_1,l_2\ge \max\{p+2, \lceil\frac{\kappa+1}{2}\rceil\}$. 
Then 
 $$d(Q_1+\calC_{2^p}, Q_2 + \calC_{2^p}) > 0.$$
\end{lemma}
\begin{proof} 
We consider the case when one of the cubes, say $Q_1,$ lies in layer $l$ and the
other, say $Q_2,$ lies in layer $l+1$; the case where
both cubes lie in the same layer or are more than one layer apart can be handled
with cruder estimates.

The centre of a ball $B = B(c_B,r_B)$ in $\B_{2^p}$ intersecting a layer
$L_{l}$ satisfies 
$|c_B|\ge 2^{l-1}-r_B\ge 2^{l-1}-2^p|c_B|^{-1}$, which in view of Lemma
\ref{lem:llmin1} implies that
$|c_B|\ge 
2^{l-1}-  2^{p-l+2}$. Therefore 
$r_B \le 2^p/(2^{l-1}- 2^{p-l+2}).$ For $j\in\{1,2\}$ let $B_j = B(c_{B_j},
r_{B_j})$ be a ball in $\B_{2^p}$ intersecting $Q_j.$
It follows that 
$$
 r_{B_1} \le \frac{1}{2^{l-p-1}-2^{-l+2}}, \quad
 r_{B_2} \le \frac{1}{2^{l-p} - 2^{-l+1}}.$$
The cubes $Q_1$ and $Q_2$ are separated by at least $2^{\kappa}-1$ cubes in 
$\Delta_{0,l}^\gamma$ or $\Delta_{0,l+1}^\gamma$, so the distance between $Q_1$
and $Q_2$ is at least
$(2^{\kappa}-1)/2^{l+1}$. Hence, using that $l\ge p+2\ge 3$, 
\begin{align*}
d(Q_1+\calC_{2^p},Q_{2}+\calC_{2^p}) 
 & \ge \frac{2^{\kappa}-1}{2^{l+1}} 
 	- \Big( \frac{1}{2^{l-p-1}-2^{-l+2}} 
	+\frac{1}{2^{l-p} - 2^{-l+1}}\Big)
\\& \ge \frac{2^{\kappa}-1}{2^{l+1}} 
 	- \Big( \frac{1}{2^{l-p-1}-2^{-p}} 
	+\frac{1}{2^{l-p} - 2^{-p-1}}\Big)
\\& = \frac{2^{\kappa}-1}{2^{l+1}} 
 	- 2^p \Big( \frac{2}{2^{l}-2} 
	+\frac{1}{2^{l} - \frac12}\Big)
\\& \geq \frac{2^{\kappa}-1}{2^{l+1}} 
 	- 2^p \frac{3}{2^{l} - 2}
\\& \geq \frac{2^{\kappa}-1}{2^{l+1}} 
 	- 2^p \frac{8}{2^{l+1}} 
\\& = \frac{2^{\kappa}- 2^{p+3}-1}{2^{l+1}},
\end{align*}
and the right-hand side is strictly positive since $\kappa\ge p+4$.
 \end{proof}

In the remainder of this section we fix integers $p\ge 2$ 
and take $\kappa= p+4$.
Note that all $l\ge p+2$ 
then satisfy the assumptions of Lemma \ref{lem:Whit}. 

\begin{definition}\label{def:i-cube}
The set $A_{p}^{(i)}$ is the the union of all cubes in $\bigcup_{l\ge p+2
}\Delta_{0,l}^\g$ with label $i \in\{1,\dots,2^{p+4}\}^n$.
\end{definition}

\begin{definition}\label{def:adm-Whit}
Let $\lambda>0$. A set $A\subseteq \R^n$ is said to be a {\em $\lambda$-admissible
Whitney set} if for all $x\in A$
we have  
$$ d(x,\complement A)\le \lambda m(x).$$
\end{definition}

Clearly, subsets of admissible Whitney sets are admissible Whitney.

\begin{theorem}\label{thm:A-Whitney} Let $p\ge 2$. 
\begin{enumerate}
\item[\rm(i)] $Q+\calC_{2^p}$ with $Q\in \Delta_{0,l}^\g$ and $l=0,\ldots,p+1$
is $2^{2p+2}\sqrt{n}$-admissible Whitney; 
\item[\rm(ii)] $A_{p}^{(i)}+\calC_{2^p}$ with 
$i\in\{1,\dots,2^{p+4}\}^n$ is $2^{p+3}\sqrt{n}$-admissible Whitney.
\end{enumerate}
\end{theorem}

\begin{proof}
In view of Lemma \ref{lem:Whit}
both assertions follow from the 
fact (proved next) that $Q+\calC_{2^p}$ is admissible Whitney for any cube
$Q\in\Delta_{0,l}^\g$, with constant
$2^{2p+2}\sqrt{n}$ for $l=0,\ldots,p+1$ and constant 
$2^{p+3}\sqrt{n}$ for $l\ge p+2$.

First let $l\in\{0,\ldots,p+1\}$.
Let $Q\in \Delta_{0,l}^\g$, consider a ball $B=B(c_B,r_B) \in
\B_{2^p}$ intersecting $Q$, and remark that 
$r_B \le 2^p$. It follows that $$ Q+\calC_{2^p}\subseteq
\{z\in\R^n: \ d(z,Q) \leq 2^p\}.$$
Let $z\in Q+\calC_{2^p}$ be given.
If $z\in Q$, then the distance of $z$ to the complement of $ Q+\calC_{2^p}$ is
at most $\frac{1}{2}+2^p.$ 
At the same time, $m(z)\ge \frac{1}{2^{p+1}\sqrt{n}}$ (since $z\in L_0\cup
\cdots\cup L_{p+1}$).
If $z\not\in Q$, then the distance of $z$ to the complement of $ Q+\calC_{2^p}$
is at most $2^{p}$. 
At the same time, $m(z)\ge \frac{1}{2^{p+2}\sqrt{n}}$ (since $z\in L_0\cup
\cdots\cup L_{p+2}$). 
In both cases, the inequality in Definition \ref{def:adm-Whit} is satisfied.

Next let $l\ge p+2$. Let $Q\in \Delta_{0,l}^\g$ be given and consider a ball $B=B(c_B,r_B)$ in $\B_{2^p}$ intersecting $Q$. 
Using Lemma \ref{lem:llmin1} we find that
$r_B\le
2^{p} |c_B|^{-1}\le  2^{p-l+2}$. It follows that $$ Q+\calC_{2^{p}}\subseteq
\{z\in\R^n: \ d(z,Q) \leq  2^{p-l+2}\}.$$
Now let $z\in Q+\calC_{2^{p}}$ be given.
If $z\in Q$, then the distance of $z$ to the complement of $ Q+\calC_{2^{p}}$ is
at
most $2^{-l-1}  
+ 2^{p-l+2}$. At the same time, $m(z)\ge \frac{1}{2^{l}\sqrt{n}}$ (since $z\in
L_l$).
If $z\not\in Q$, then the distance of $z$ to the complement of $
Q+\calC_{2^{p}}$ is
at most $2^{p-l+2}$. At the same time, $m(z)\ge \frac{1}{2^{l+1}\sqrt{n}}
$ (since $z\in L_{l-1}\cup L_l\cup L_{l+1}$). 
In each of these cases, the inequality in Definition \ref{def:adm-Whit} is
satisfied.
\end{proof}

\begin{corollary}
\label{thm:Whitney}
There exists a constant $N$, depending only on $p\ge 2$ and the dimension
$n$,
such that every open set in $\R^n$ can be covered by 
$N$ open $2^{2p+2}\sqrt{n}$-admissible Whitney sets.
\end{corollary}

An explicit bound on $N$ is obtained by 
counting the number of sets involved in
Theorem \ref{thm:A-Whitney}, which can be estimated by $2^n(1+
2^{(p+4) n} + \ldots + 2^{(p+1) (p+4) n}) + 2^{(p+4) n}$.

The next result is an immediate consequence of its euclidean counterpart
(see \cite[VI.1]{littleStein} for the details). The cubes that we pick up from
the euclidean proof will automatically be admissible at a suitable scale (which
depends upon $n$ only) because we start from an admissible Whitney set.

\begin{lemma}\label{lem:Wpart}
Let $\lambda>0$ and suppose $O\subseteq \R^n$ is an open $\lambda$-admissible Whitney set.
There exists a constant $\rho$, depending only on $\lambda$ and the
dimension
$n$, a countable family of disjoint cubes $(Q_{m})_m$ in $\Delta^{\gamma}$, and
a
family of functions $(\phi_{m})_m \subseteq C_{\rm c}^\infty(\R^{n})$
such that:
\begin{enumerate}
\item[\rm(i)]
$\bigcup_{m} Q_{m} = O$;
\item[\rm(ii)]
for all $m$ we have $\diam(Q_{m}) \leq d(Q_{m},\complement O) \leq
\rho\diam(Q_{m})$;
\item[\rm(iii)]
for all $m$ we have $\supp(\phi_{m}) \subseteq Q_{m}^{*}$, where $Q_{m} ^{*}$
denotes the cube with the same center as $Q_{m}$ but side
length multiplied by $\rho$;
\item[\rm(iv)]
for all $m$ and all $x \in Q_{m}$ we have $\frac1\rho \leq
\phi_{m}(x) \leq 1$;
\item[\rm(v)]
for all $x \in O$ we have $\sum  _{m} \phi_{m}(x) =
1.$
\end{enumerate}
\end{lemma}

\section{Gaussian tent spaces}

Throughout this section we fix $1 < q < \infty$ and let $q' := \frac{q}{q-1}$
denote its conjugate exponent.
Let
$$D := \{(x,t)\in \R^n\times (0,\infty): \ t < m(x)\}.$$  
Note that a point $(x,t)\in \R^d\times (0,\infty)$ belongs to 
$\overline D$ if and only if $B(x,t)\in \B_1$.
 
\begin{definition} 
The {\em Gaussian tent space} $T^{1,q}(\gamma)$ is the completion of $C_{\rm
c}(D)$ with respect to the norm
$$
\|f\|_{T^{1,q}(\gamma)} := \|Jf\|_{L^{1}(\R^n, d\gamma(x);L^{q}(D,
d\gamma(y)\frac{dt}{t}))},
$$
where 
$$
(J f(x))(y,t) := \frac{\one_{B(y,t)}(x)}{{\gamma(B(y,t))}^{\frac{1}{q}}}f(y,t). $$
\end{definition}

For a measurable set $A\subseteq \R^n$ and a real number $\a>0$ 
we define the {\em tent with aperture $\a$ 
over $A$} by
$$ T_\a(A)= \{(y,t) \in \R^n\times (0,\infty): \  d(y,\complement A) \geq \a
t\}.$$

\begin{definition}\label{def:atom} Let $\a$ be a positive real number.
A function $a : D\to \C$ is called a
$T^{1,q}(\gamma)$ {\em $\a$-atom} if there exists a ball
$B \in \B_{\a}$ such that
\begin{enumerate}
\item[\rm(i)]
$a$ is supported in $T_1(B)\cap D$; 
\item[\rm(ii)]
$\|a\|_{L^{q}(D,\frac{d\gamma dt}{t})} \le
{{\gamma(B)}^{-1/q'}}$.
\end{enumerate}
\end{definition}

\begin{lemma}
If $a$ is a $T^{1,q}(\gamma)$ $\a$-atom, then $a \in
T^{1,q}(\gamma)$ and $\n a\n_{T^{1,q}(\gamma)}\leq 1.$
\end{lemma}
\begin{proof} Let $a$ be a $T^{1,q}(\gamma)$ $\a$-atom supported in $T_1(B)\cap
D$ for some $B \in \B_\a$.
If $(y,t)\in T_1(B)\cap D$ 
and $x\in B(y,t)$, then $x\in B$. First using this fact, then
H\"older's inequality, then the Fubini theorem, we obtain
\begin{equation*}
\begin{split}
\int  _{\R^{n}}& \Big(\iint _{D}
\frac{\one_{B(y,t)}(x)}{\gamma(B(y,t))}
|a(y,t)|^{q}d\gamma(y)\frac{dt}{t}\Big)^{\frac{1}{q}}d\gamma(x) \\
&= \int  _{\R^{n}} \Big(\iint_{D}
\frac{\one_{B(y,t)}(x)}{\gamma(B(y,t))}
 |a(y,t)|^{q}d\gamma(y)\frac{dt}{t}\Big)^{\frac{1}{q}} \one_B(x) d\gamma(x) \\
& \leq   \Big(\int_{\R^n}\iint_{D}
\frac{\one_{B(y,t)}(x)}{\gamma(B(y,t))}
|a(y,t)|^{q}d\gamma(y)\frac{dt}{t}\,d\gamma(x)\Big)^\frac1q
{\gamma(B)}^{\frac{1}{q'}}
\\ & =  \Big(\iint_{D}
|a(y,t)|^{q}d\gamma(y)\frac{dt}{t}\Big)^\frac1q{\gamma(B)}^{\frac{1}{q'}}
\\ & \leq 1.
\end{split}
\end{equation*}
\end{proof}

The set $D$ admits a locally finite cover with tents $T_1(B)$ based at balls 
$B\in\B_\a$ 
if and only if $\a>1$; this explains the condition $\a>1$ in the next theorem, 
which establishes an atomic decomposition of $T^{1,q}(\gamma)$.
The proof follows the lines of the euclidean counterpart in \cite{CMS} (see also
the expanded version in the
setting of spaces of homogeneous type \cite{R}). However, one
needs to be careful not to use a doubling property for
non-admissible balls; it is here where the results of the previous section
come to rescue.

\begin{theorem}[Atomic decomposition]\label{thm:atomic}
For all $f \in T^{1,q}(\gamma)$ and $\a>1$, there exist a sequence 
$(\lambda_{n})_{n\ge 1}
\in \ell_{1}$ and a sequence of $T^{1,q}(\gamma)$ $\a$-atoms $(a_{n})_{n\ge 1}$
such that
\begin{enumerate}
\item[\rm(i)] $f = \sum  _{n\ge 1}\lambda_{n}a_{n}$;
\item[\rm(ii)]
$\sum  _{n\ge 1}|\lambda_{n}| \lesssim
\|f\|_{T^{1,q}(\gamma)}$.
\end{enumerate}
\end{theorem}
Before we start with the proof, we need some
notations and auxiliary results. Given a measurable set $A\subseteq \R^n$
and a real number $\alpha>0$, we define
\begin{equation*}
R_{\alpha}(A) = \{(y,t) \in \R^d\times (0,\infty): \ 
d(y,A) < \alpha t\} =  \complement T_{\alpha}(\complement A).
\end{equation*}
We also put, for any measurable set $A\subseteq \R^n$ and real
number $\beta>0$,
$$
A^{[\beta]} = \Big\{x \in \R^{n} : \frac{\gamma(A\cap
B)}{\gamma(B)} \geq \beta\, \hbox{ for all } B \in \B_{\frac32}\hbox{ with
centre $x$}\Big\}.
$$
We call $A^{[\beta]}$ the set of points of admissible
${\beta}$-density of $A$. Note that $A^{[\beta]}$ is a
closed subset of $\R^n$ contained in $\overline{A}$.

\begin{lemma} \label{lem:RFeta-estimate}
For all $\eta \in (\frac12,1)$ there exists an $\overline\eta \in
(0,1)$ such that, for all measurable sets $F \subseteq \R^{n}$ and
all non-negative measurable functions $H$ on $D$,
$$
\iint_{R_{1-\eta}( F^{[\overline\eta]})\cap D} H(y,t)\,d\gamma(y)\frac{dt}{t} 
\lesssim \int  _{F} \iint_{D}\frac{\one_{B(y,t)}(x)}{\gamma(B(y,t))}H(y,t)\,
d\gamma(y)\frac{dt}{t}d\gamma(x).
$$
\end{lemma}
\begin{proof} First let $\overline\eta\in (0,1)$ be arbitrary and fixed.
Let $(y,t) \in R_{1-\eta}(F^{[\overline\eta]})\cap D$. Note that $(y,t)\in D$
implies $B(y,t)\in\B_1$. There exists $x \in
F^{[\overline\eta]}$
such that $|y-x|<(1-\eta)t$. Notice first that, since $t \leq m(y),$ we have
$|x| <
(1-\eta)t+\frac{1}{t} \leq \frac12 + \frac1t \le \frac32 \frac1{t}$. We thus
have that $t
\in (0,\frac32 m(x))$. Moreover $B(x,\eta t) \subseteq B(y,t)
\subseteq B(x,\frac32 t)$, and thus $B(y,t) \in \B_1$,
$B(x,t) \in \B_{\frac32}$, and $\gamma(B(x,t)) \eqsim
\gamma(B(y,t))$ by repeated application of the doubling
property on admissible balls (Lemma \ref{lem:doubling}). We therefore have
\begin{equation*}
\begin{split}
\gamma(F \cap B(y,t)) &\geq \gamma(F\cap B(x,t))-\gamma(B(x,t)\cap \complement
B(y,t)) \\
& \geq \overline{\eta}\gamma(B(x,t))-\gamma(B(x,t))+\gamma(B(x,t)\cap B(y,t)) \\
& \geq (\overline{\eta}-1)\gamma(B(x,t))+\gamma(B(x,\eta t)).
\end{split}
\end{equation*}
Now, picking $\overline{\eta}$ close enough to $1$ and using
the doubling property, we obtain
a constant $c=c(\eta, n) \in (0,1)$ such that
$$\gamma(F\cap B(y,t)) \geq c\gamma(B(x,t)).$$
Therefore, there exists a constant $c'= c'(\eta,n)>0$ such that
$ \gamma(F\cap B(y,t)) \geq c'\gamma(B(y,t))$ for all
$(y,t)\in R_{1-\eta}(F^{[\overline\eta]})\cap D.$ Finally,
\begin{equation*}
\begin{split}
\int  _{F} \iint_{D} \frac{\one_{B(y,t)}(x)}
{\g(B(y,t)}H(y,t)\,d\gamma(y)\frac{dt}{t}d\gamma(x) 
&= \iint_{D} \frac{\gamma(F\cap B(y,t)
)}{\g(B(y,t)}H(y,t)\,d\gamma(y)\frac{dt}{t} \\
& \ge c'
\iint_{R_{1-\eta} (F^{[\overline\eta]})\cap D}H(y,t)\,d\gamma(y)\frac{dt}{t}.
\end{split}
\end{equation*}
\end{proof}
  
\begin{lemma}\label{lem:atomic-alpha}
If a function $f\in  T^{1,q}(\g)$ admits a decomposition
in terms of $ T^{1,q}(\g)$ $\a$-atoms for some $\a>1$, 
then it admits a decomposition
in terms of $ T^{1,q}(\g)$ $\a$-atoms for all $\a>1$.
\end{lemma}
\begin{proof}
Suppose that $f\in  T^{1,q}(\g)$ admits a decomposition
in terms of $T^{1,q}(\g)$ $\beta$-atoms for some $\beta>1.$ We will show that
$f$ admits a decomposition in terms of $ T^{1,q}(\g)$ $\alpha$-atoms for any
$\alpha > 1$. 
This is immediate if $\alpha \geq \beta,$ since in this case any  $T^{1,q}(\g)$
$\beta$-atom is a $T^{1,q}(\g)$ $\alpha$-atom as well. 

Let us now assume that 
$1 < \alpha < \beta.$ We claim that it suffices to show that there exists an integer 
$N,$ depending only upon $\alpha,\beta,$ and the dimension $n$, such that if 
$B\in \B_\b$, then $T_1(B)\cap D$
can be covered by at most $N$ tents of the form $T_1(B')$ with $B' = B(c',r')
\in \B_\a$ satisfying $r'=\a m(c')$. 

To prove the claim, it clearly suffices to consider the case that $f$ is a 
$T^{1,q}(\gamma)$ $\beta$-atom having support in $T_1(B) \cap D$ for some ball 
$B \in \B_\beta$ with center $c$ and radius $r = \beta m(c).$ Let $\{T_1(B_1'), 
\ldots, T_1(B_N')\}$ be a covering of $T_1(B),$ where each $B_j',$ $j=1, \ldots, N,$ 
is a ball in $\B_\alpha$ with center $c_{j}$, radius $r_j = \alpha m(c_j)$, and intersecting $B$. 
For $x \in T_1(B)\cap D$ we set  $n(x) := \#\{ 1 \leq j\leq N : x \in T_1(B_j')\}$ 
and $f_j(x) := n^{-1}(x) f(x)\one_{T_1(B_j')}(x).$ It then follows that 
$f= \sum_{j=1}^N f_j.$ Moreover, each $f_j$ is a $T^{1,q}(\gamma)$ $\alpha$-atom, 
since $f_j$ is supported in $T_1(B_j) \cap D$ and 
\begin{align*}
 \|f_j\|_{L^q(D, \frac{d\gamma \,dt}{t})}
   \leq  \|f\|_{L^q(D, \frac{d\gamma \,dt}{t})}
   \leq \gamma(B)^{-1/q'}
   \lesssim \gamma(B_j')^{-1/q'}.
\end{align*}
To obtain the latter estimate, we pick an arbitrary $b \in B_{j}'\cap B$ and use Lemma \ref{lem:mnp1}(ii) to conclude that $m(c_j) \leq (1+\alpha) m(b) \leq 2(1+\alpha)(1+\beta) m(c)$,
and then we estimate
\begin{align*}
 r_j = \alpha m(c_j) 
 	\leq 2\alpha(1+\alpha) (1+\beta) m(c) 
	   = 2\alpha  \beta^{-1}(1+\alpha)(1+\beta) r.
\end{align*}
Combined with Lemma \ref{lem:doubling}, we infer that $\gamma(B_j) \lesssim \gamma(B).$ 
It follows that $f = \sum_{j=1}^N f_j$ is a decomposition in terms of 
$T^{1,q}(\g)$ $\a$-atoms, which proves the claim.

Fix $R>1+ \beta$ so large that 
$\alpha(R-\beta)/(R-\beta+\alpha)>1$. The set $\{(y,t) \in D : |y| \leq
R\}$ can be covered with finitely many sets -- their number depending only
upon $R,$ $n$ and $\alpha$ -- of the form $T_1(B')$ with $B'=B(c',r') \in
\B_\alpha$ and $r'=\a m(c')$.

Take a ball $B = B(c,r) \in \B_\beta$ with $|c|\ge R$ and choose
$\delta \in (0,1)$ so small that
$(1-\delta)\alpha(R-\beta)/(R-\beta+\alpha)>1$.
We first remark that, if $x \in B$, then ${|x|} \geq {R-\beta}
\geq 1$, and therefore $m(x) = \frac{1}{|x|}.$
Let us then define 
$$
C_{B} := \{(x,t)\in B\times (0,\infty) \;:\; t\leq m(x)\}.
$$
Noting that $T_1(B)\cap D\subseteq C_B$, it remains to cover $C_{B}$ 
with $N$ tents $T_{1}(B')$ based on balls $B' \in
\B_{\alpha}$, where the number $N$ depends on $\a$, $\b$ and $n$ only. 

To do so, let us start by picking $c'\in B$, and let
$r'=\alpha m(c') = \frac{\alpha}{|c'|}$ and $B'=B(c',r')$.
If $(x,t) \in C_{B}$ is such that $|x-c'|\leq \delta r'$, then
\begin{align*}
d(x,\complement B') & = d(c',\complement B')-|x-c'| \geq
(1-\delta)r' = (1-\delta)\frac{\a}{|c'|} \\
& \geq (1-\delta) \frac{\a}{|x|+|x-c'|} \geq
m(x)(1-\delta)(\frac{\a|x|}{|x|+\alpha}) \\
&\geq m(x)(1-\delta)\frac{\a(R-\beta)}{R-\beta+\alpha} \geq m(x) \geq t.
\end{align*}
Here we used the monotonicity of the function $t\mapsto t/(t+\a)$.

We have proved that a point $(x,t) \in C_{B}$ belongs to $T_{1}(B')$ 
whenever $|x-c'|\leq \delta r'$.
Using that $(|c|+ \beta)r \le (|c|+ \beta)\frac{\b}{|c|} \leq \b + \beta^2$, we have
$$
r'= \frac{\alpha}{|c'|} \geq \frac{\alpha}{|c|+\beta} \geq
\frac{\alpha}{\b+\beta^2}r.
$$
This implies that $B$ can be covered with $N$ balls $B'=B(c',\delta r')$
as above, with $N$
depending only on $\alpha,$ $\beta$ and $n$.
The union of the $N$ sets $T_1(B')\cap D$ will then cover $C_B$, thus completing 
the proof. 
\end{proof}

\begin{proof}[Proof of Theorem \ref{thm:atomic}]
By Lemma \ref{lem:atomic-alpha} it suffices to prove that each
$f\in  T^{1,q}(\g)$ admits 
a decomposition in terms of $ T^{1,q}(\g)$ $\a$-atoms for {\em some} $\a>0$.

Recall that the disjoint sets $A_{p}^{(i)}$ have been introduced in Definition
\ref{def:i-cube}. We shall apply Theorem \ref{thm:A-Whitney} with $p=4$ 
(the reason for this choice is the constant $16 = 2^4$ produced in the argument below).
Since  
$$\big(\bigcup_{0\le l\le 5} L_l\big) \cup\big(\bigcup_{i\in
\{1,\dots,2^8\}^n}A_{4}^{(i)}\big) = \R^n$$ we may write 
\begin{align}\label{eq:dec-f} 
f = f \one_{\{\n Jf\n_q >0\}} = \sum_{0\le l\le 5}\sum_{Q\in \Delta_{0,l}^\g}
f \one_{Q\cap \{\n Jf\n_q >0\}} + 
\sum_{i\in \{1,\dots,2^8\}^n} f \one_{{A_{4}^{(i)}}\cap\{ \n Jf\n_q >0\}},
\end{align}
where
$f\one_{\{\n Jf\n_q >0\}}(x,t) := f(x,t)\one_{\{\n Jf\n_q >0\}}(x)$ and $\{\n Jf\n_q >0\} = \{x\in\R^n: \ \n
Jf(x)\n_{L^q(D,d\gamma(y)\frac{dt}{t})} > 0\}$. 
The first equality in \eqref{eq:dec-f}, which holds almost everywhere on $D$, 
is justified as follows.
For all $x\in V:= \{\n Jf\n_q=0\}$ 
we have $\one_{B(y,t)}(x)f(y,t) = 0$ for almost all $(y,t)\in D$, and
therefore,
by Fubini's theorem, for almost all $y\in\R^d$ we have  
$\one_{B(y,t)}(x)f(y,t) = 0$ for almost all $t>0$.
Fix $\d>0$ arbitrary.  Then for almost all $y\in B(x,\d)$ 
we have $f(y,t) = 0$ for almost all $t\ge \d$. By another application of
Fubini's 
theorem this implies
that $f(y,t)=0$ for almost all $(y,t)\in (B(x,\d)\times [\d,\infty))\cap
D$.
Taking the union over all rational $\d>0$, it follows that $f\equiv 0$ almost
everywhere on 
$\Gamma_{x} := \{(y,t)\in D: \ |x-y|<t\}$, the `admissible cone' over
$x$. 
If $K$ is any compact set contained in $V$, then by taking the union over a 
countable dense set of points $x\in K$ it follows 
that $f(y,t)=0$ almost everywhere on the `admissible cone' over $K$. Finally,
by the inner regularity of the Lebesgue measure on $\R^n$, 
it follows that $f(y,t)=0$ almost everywhere on the `admissible cone' over
$V$. In particular this gives
$f(x,t)=0$ for almost all $(x,t)\in D$ with $x\in K$. This proves the first 
identity in \eqref{eq:dec-f}.  

To prove the theorem it suffices to prove that each of the summands on the
right-hand side of \eqref{eq:dec-f}
has an atomic decomposition. In view of Theorem \ref{thm:A-Whitney}
(applied with $p=4$) it even
suffices to prove that
$$g:= f \one_{W\cap \{\n Jf\n_q >0\}}$$ has an atomic decomposition for any given
measurable set $W$ in $\R^n$ such that $W+\calC_{16}$ is $2^{10}\sqrt{n}$-admissible Whitney. 

Given $k\in\Z$, let us define 
$$ O_{k} := \{\|Jg\|_{q}>2^{k}\}$$ and $ F_{k} := \complement  O_{k}$. Fix an
arbitrary $\eta \in (\frac12,1)$ 
and let $\overline{\eta}$ be as
in Lemma \ref{lem:RFeta-estimate}. With abuse of notation we let
$ O_{k}^{[\overline\eta]} := \complement  F_{k}^{[\overline\eta]}$,
where $ F_{k}^{[\overline\eta]}$ denotes the set of points of
admissible $\overline{\eta}$-density of $F_{k}$, and note that 
$O_{k}\subseteq O_{k}^{[\overline\eta]}$.
We claim that  $ O_{k}^{[\bar\eta]}$ is contained in $W+\calC_{16}$ (see
\eqref{eq:AplusB}). 

To prove the claim we first fix $x\in O_{k}$ 
and check that $x\in W+\calC_2$.
Indeed, since $Jg(x)$ does not vanish almost everywhere on $D$ 
we can find a set $D'\subseteq D$ of positive measure such that for almost all 
$(y,t)\in D'$ one has $\one_{B(y,t)}(x)g(y,t) =  
\one_{B(y,t)}(x)f(y,t)\one_{W\cap\{\n Jf\n_q>0\}}(y) \not=0$. 
For those points we have $y\in W$, 
$|x-y|<t$ and $t<m(y)$, so $t<2m(x)$ by Lemma \ref{lem:mnp1} (i).
Thus $B(x,t)$ belongs to $\B_2$ and intersects $W$, so $x\in W+\calC_2$.

Next let $x\in O_{k}^{[\overline\eta]}$.
Then $x$ is not a point of admissible $\overline{\eta}$-density of $F_{k}$,
so there is a ball $B\in \B_\frac32$ with centre $x$ such that 
$\g(F_{k}\cap B) < \bar\eta \g(B).$
This is only possible if $B$ intersects $O_{k} = \complement F_{k}$.
Since $O_{k}$ is contained in $W+\calC_2$, this
means that $B$ intersects
$W+\calC_{2}$. Fix an arbitrary $x'\in B\cap (W+\calC_{2})$
and let $B' \in \mathcal{B}_{2}$ be any admissible ball centred at $x'$ and 
intersecting $W$.
From $x'\in B$ and $B \in \B_{\frac32}$ it follows that $|x-x'| <\frac32m(x)$. Also, 
since $B'$ belongs to $\B_2$ 
and intersects $W$, $d(x',W)<2m(x')$. It follows that $d(x,W)<\frac32 m(x)+2m(x')$.
By the second part of Lemma \ref{lem:mnp1}(ii) 
we have $m(x')\le 5 m(x)$  
and therefore 
$\dist(x,W) \le \frac{23}{2}m(x)$.
This proves the claim (with a 
somewhat better constant, but that is irrelevant).

For each $N\ge 1$ define $g_N(y,t) := \one_{\{|y|\le
N\}}\one_{\{|g|\le N\}} \one_{(\frac{1}{N}, \infty)}(t) g(y,t)$. Clearly, $
g_N\in T^{1,q}(\g)$
and, by dominated convergence, $\lim_{N\to\infty} g_N = g$ in
$ T^{1,q}(\g)$. Defining the sets $ F_{k,N}$, $ O_{k,N}$,
$ F_{k,N}^{[\overline\eta]}$, $ O_{k,N}^{[\overline\eta]}$ in
the same way as above, Lemma \ref{lem:RFeta-estimate} gives that
\begin{equation*}
\begin{split}
\iint_{R_{1-\eta}( F_{k,N}^{[\overline\eta]})\cap D} & |g_N(y,t)|^{q}
d\gamma(y)\frac{dt}{t} \\
& \lesssim \int  _{ F_{k,N}} \iint_{D}
\frac{\one_{B(y,t)}(x)}{\gamma(B(y,t))} |g_N(y,t)|^{q}
d\gamma(y)\frac{dt}{t}d\gamma(x)  \lesssim
\n g_N\n_{ T^{1,q}(\g)}^q.
\end{split}
\end{equation*}
As $k\to-\infty$, the middle term tends to $0$ and
therefore the support of $g_N$ is contained in the union
 $\bigcup
_{k \in \Z} T_{1-\eta}( O_{k,N}^{[\overline\eta]})\cap D. $ Clearly,
$ O_{k,N}\subseteq  O_k$ implies
$T_{1-\eta}( O_{k,N}^{[\overline\eta]})\subseteq
T_{1-\eta}( O_{k}^{[\overline\eta]})$, and therefore a limiting
argument shows that the support of $g$ is contained in the union $\bigcup _{k
\in \Z} T_{1-\eta}( O_{k}^{[\overline\eta]})\cap D$.

Choose cubes $(Q_{k}^{m})_{m}$ 
and functions $(\phi_{k}^{m})_{m}$
as in Lemma \ref{lem:Wpart}, applied to the open sets $ O_k^{[\overline\eta]},$ 
which are contained in $W+\calC_{16}$.
Define for $(y,t) \in D,$
\begin{equation*}
\begin{split}
b_{k}^{m}(y,t) & :=
 (\one_{T_{1-\eta}( O_k^{[\overline\eta]})}(y,t)
 -\one_{T_{1-\eta}( O_{k+1}^{[\overline\eta]})}(y,t))\phi_{k}^{m}(y) g(y,t),\\
\mu_{k}^{m} &:= \iint_{D}
|b_{k}^{m}(y,t)|^q\,d\gamma(y)\frac{dt}{t},\\
\end{split}
\end{equation*}
and put
\begin{equation*}
\lambda_{k}^{m} :=
(\gamma(Q_{k}^{m}))^{\frac{1}{q'}}(\mu_{k}^{m})^{\frac{1}{q}}, \quad
a_{k}^{m}(y,t) := \frac{b_{k}^{m}(y,t)}{\lambda_{k}^{m}}.
\end{equation*}
Then,
$$g = \sum  _{k \in \Z} \sum  _{m}
\lambda_{k}^{m}a_{k}^{m}.
$$

Let $C$ be a constant to be determined later and denote by
$(Q_{k}^{m})^{**}$ the cube which has the same center as
$Q_{k}^{m}$ but side-length multiplied by $C$. Let us further
denote by $\ell_k^m$ and $\d_{k}^{m}$ the sidelength and the 
length of the diagonal of
$Q_{k}^{m}$, respectively, and by $c_{k}^{m}$ the center of $Q_k^m$. 
We claim that
$$\supp(a_{k}^{m}) \subseteq T_{1}((Q_{k}^{m})^{**}).$$ 
We have
$$\supp(a_{k}^{m}) \subseteq T_{1-\eta}( O_k^{[\overline\eta]}) \cap \{(y,t) \in
D: \  y \in (Q_{k}^{m})^{*}\},$$
where $(Q_{k}^{m})^{*}$ is as in Lemma \ref{lem:Wpart}.
Therefore, fixing $(y,t) \in \supp(a_{k}^{m})$, we have
$d(y,F_k^{[\overline\eta]}) \geq (1-\eta)t$ and $y \in
(Q_{k}^{m})^{*}$. For $z \not \in (Q_{k}^{m})^{**}$ this gives
$d(z,c_k^m)\ge \frac12C\ell_k^m  = \frac12 Cd_k^m /\sqrt{n}$ 
 and
\begin{equation}\label{eq:determineC}
d(y,z) \geq d(z,c_{k}^{m})-d(y,c_{k}^{m}) \geq
\big(\frac{C}{\sqrt{n}}-\rho\big)\frac{\d_{k}^{m}}{2},
\end{equation}
where $\rho=\rho_{2^{10} \sqrt{n},n}$ 
is the constant from Lemma \ref{lem:Wpart}.
Moreover, by property (ii) in Lemma \ref{lem:Wpart},
$$
d(c_{k}^{m}, F_k^{[\overline\eta]}) \leq (\rho + \frac12) \d_{k}^{m}.
$$
For $u \in  F_k^{[\overline\eta]}$ such that $d(c_{k}^{m},u)
\leq (\rho + \frac12) \d_{k}^{m} + \e$, this gives
\begin{equation}\label{eq:determineC2}
(1-\eta)t \leq d(y,F_k^{[\overline\eta]})\le  d(y,u) \leq
d(y,c_{k}^{m})+d(c_{k}^{m},u)
 \leq\frac{3\rho + 1}{2} \d_k^m+ \e.
\end{equation}
Upon taking 
$C=2\sqrt{n}(\frac{\rho}{2}+\frac{3 \rho + 1}{2(1-\eta)})$, 
from \eqref{eq:determineC} and \eqref{eq:determineC2} (where we let
$\e\downarrow 0$) we infer that $$d(y,z) \ge 
\frac{3\rho +1}{2(1-\eta)}\delta_k^m \ge t.$$
This means that $(y,t)\in T_{1}((Q_{k}^{m})^{**})$, thus proving 
the claim.

Using the definitions of $\l_{k}^{m}$ and $a_{k}^{m}$ together
with the doubling property for admissible balls, we also get
that
$$
\iint_{D}|a_{k}^{m}(y,t)|^{q}d\gamma(y)\frac{dt}{t} \leq
\frac{1}{\gamma(Q_{k}^{m})^{\frac{q}{q'}}} \lesssim
\frac{1}{\gamma((Q_{k}^{m})^{**})^{\frac{q}{q'}}}.
$$
Up to a multiplicative constant, the $a_{k}^{m}$ are thus $T^{1,q}(\gamma)$
$\a$-atoms for some $\a = \a(C,n)>0$. To get the
norm estimates, we first use Lemma \ref{lem:RFeta-estimate}. Noting that
$(y,t)\in {\rm supp}(b_k^m)$ implies 
$(y,t) \not\in T_{1-\eta}( O_{k+1}^{[\overline\eta]})$ and hence
 $(y,t)\in R_{1-\eta}( F_{k+1}^{[\overline \eta]})$, and that   
$(y,t)\in T_{1}((Q_{k}^{m})^{**})$ and $x\in B(y,t)$ imply
$x\in(Q_{k}^{m})^{**},$ we obtain
\begin{align*}
\mu_{k}^{m} &\leq
\iint_{R_{1-\eta}( F_{k+1}^{[\overline \eta]})\cap D}
\one_{T_{1}((Q_{k}^{m})^{**})}(y,t)|g(y,t)|^{q}d\gamma(y)\frac{dt}{t} \\
&\lesssim \int  _{ F_{k+1}}
 \iint_{D}\!\!
\frac{\one_{B(y,t)}(x)\one_{T_{1}((Q_{k}^{m})^{**})}(y,t)}{\gamma(B(y,t))}|g(y,
t)|^{q}d\gamma(y)\frac{dt}{t}\,d\gamma(x) \\
& \leq \int
_{ F_{k+1}\cap(Q_{k}^{m})^{**}}\| J g(x)\|_{L^q(D,d\gamma\frac{dt}{t}))}
^{q} d\gamma(x) 
\\ & \le 2^{q(k+1)}\gamma((Q_{k}^{m})^{**})
\lesssim 2^{qk}\gamma(Q_{k}^{m}).
\end{align*}
This then gives
$$
   \sum  _{k \in \Z} \sum  _{m}
\lambda_{k}^{m} =
   \sum  _{k \in \Z}\sum  _{m}
(\mu_{k}^{m})^\frac1q(\gamma(Q_{k}^{m}))^{\frac1{q'}}
    \lesssim \sum  _{k \in
\Z}2^{k}\gamma(O_k^{[\overline\eta]}).$$ 
Since $
x\in O_k^{[\overline\eta]}$ implies $ M_{\frac32}(\one_{ O_{k}})(x) 
\ge 1-\overline{\eta},$ the weak type 1-1
of the Hardy-Littlewood maximal function $M_{\frac32}$ defined by
using only $\B_{\frac32}$-balls (which is proved by copying its
euclidean counterpart, see \cite[Theorem 3.1]{MM}), gives 
$(1 - \overline \eta)\gamma( O_k^{[\overline\eta]}) \lesssim
\gamma( O_{k})$ and thus
$$
(1-\overline\eta)\sum  _{k \in \Z} \sum  _{m}
\lambda_{k}^{m} 
\lesssim \sum_{k\in\Z} 2^{k}\gamma( O_{k})
 \lesssim \int  _{0} ^{\infty} \gamma(x \in \R^{n}: \ 
\|  J g(x)\|_{q} > s)\,ds = \|g\|_{ T^{1,q}(\gamma)}.
$$
\end{proof}

As an application of the atomic decomposition we prove next a change of aperture
theorem.
Our proof is different from the euclidean proofs in \cite{CMS} and \cite{HTV} in
that we derive the result directly from the atomic decomposition theorem.

\begin{definition} For $\a>0$,
the {\em Gaussian tent space} $T_\a^{1,q}(\gamma)$ with {\em aperture $\a$} 
is the completion of $C_{\rm c}(D)$ with respect to the norm
$$
\|f\|_{T_\a^{1,q}(\gamma)} := \|J_\a f\|_{L^{1}(\R^n, d\gamma(x);L^{q}(D,
d\gamma(y)\frac{dt}{t}))},
$$
where $$
J_\a f(x,y,t) := \frac{\one_{B(y,\a
t)}(x)}{{\gamma(B(y,t))}^{\frac{1}{q}}}f(y,t), \quad f\in C_{\rm c}(D). $$
\end{definition}

\begin{theorem}[Change of aperture]\label{thm:aperture} 
For all $1 < \a_0 < \a$ we have $T_\a^{1,q}(\g) = T_{\a_0}^{1,q}(\g)$ with 
equivalent norms.
\end{theorem}

\begin{proof} 
It is clear that $T_{\a}^{1,q}(\g) \subseteq T_{\a_0}^{1,q}(\g)$, so it 
suffices to show that $T_{\a_0}^{1,q}(\g) \subseteq T_\a^{1,q}(\g)$.
For this  it suffices to show that $J_\a f \in 
L^1(\R^n, d\g(x); L^q(D,d\gamma(y)\frac{dt}{t}))$ whenever $f\in T_{\a_0}^{1,q}(\g)$. 
Now, by the doubling property
(noting that $(y,t)\in D$ implies $B(y,t)\in\B_1$),
\begin{align*} 
\n J_\a f\n_{L^1(\R^n,d\g(x); L^q(D,  \frac{d\g(y)dt}{t}))} 
& = 
\int_{\R^n} \Big(\int_D \frac{\one_{B(y,\a t)}(x)}{{\gamma(B(y,t))}}|f(y,t)|^q
d\gamma(y)\frac{dt}{t}\Big)^\frac1q d\g(x)
\\ & = 
\int_{\R^n} \Big(\int_{\tilde D} \frac{\one_{B(y,t)}(x)}{{\gamma(B(y,
t/\a))}}|f(y,t/\a)|^q d\gamma(y)\frac{dt}{t}\Big)^\frac1q d\g(x)
\\ & \lesssim
\int_{\R^n} \Big(\int_{\tilde D} \frac{\one_{B(y, t)}(x)}{{\gamma(B(y,
t))}}|f(y,t/\a)|^q d\gamma(y)\frac{dt}{t}\Big)^\frac1q d\g(x)
\\& = \n J \tilde f \n_{L^1(\R^n,d\g(x); L^q(\tilde D,
d\gamma(y)\frac{dt}{t}))},
\end{align*}
where $\tilde D := \{(y,t)\in \R^n\times (0,\infty): \ (y,\a^{-1} t)\in D\}$, and
$\tilde f(y,t) := f(y, \a^{-1}t)$. To
prove the theorem, it thus suffices to show that 
\begin{equation}
\label{eq:tildeJ}\begin{aligned}
\n J \tilde f \n_{L^1(\R^n,d\g(x); L^q(\tilde D, d\gamma(y)\frac{dt}{t}))}
 \lesssim \n J_{\a_0}  f \n_{L^1(\R^n,d\g(x); L^q(D, d\gamma(y)\frac{dt}{t}))}
\end{aligned}
\end{equation}
for $f \in T^{1,q}(\gamma).$

Suppose $a$ is a $T^{1,q}(\g)$ $\alpha_0$-atom $a.$ Then $a$ is supported in 
$T_1(B)\cap D$ for some ball $B=B(c,r) \in
\B_{\alpha_0}$. Then
$\tilde a(y,t) := a(y, \a^{-1}t)$ is supported in $\tilde T_1(B)\cap \tilde
D$, where $\tilde T_1(B):= \{(y,t)\in \R^n\times (0,\infty): \
(y,t/\a)\in T_1(B)\}$. Using that $(y,t)\in \tilde T_1(B)$ and $x\in
B(y,t)$ imply $x\in B(c,\a r)$,  the doubling property for admissible balls
gives 
\begin{align*}
\ & \int  _{\R^{n}} \Big(\iint _{\tilde D}
\frac{\one_{B(y,t)}(x)}{\gamma(B(y,t))}
|a(y,t/\a)|^{q}d\gamma(y)\frac{dt}{t}\Big)^{\frac{1}{q}}d\gamma(x) \\
& \qquad\le \int  _{\R^{n}} \Big(\iint_{\tilde D}
\frac{\one_{B(y,t)}(x)}{\gamma(B(y,t))}
 |a(y,t/\a)|^{q}d\gamma(y)\frac{dt}{t}\Big)^{\frac{1}{q}} \one_{B(c,\a
r)}(x) \,d\gamma(x) \\
& \qquad \leq   \Big(\int_{\R^n}\iint_{\tilde D}
\frac{\one_{B(y,t)}(x)}{\gamma(B(y,t))}
|a(y,t/\a)|^{q}d\gamma(y)\frac{dt}{t}\,d\gamma(x)\Big)^\frac1q
{\gamma(B(c,\a r))}^{\frac{1}{q'}}
\\ & \qquad = \Big(\iint_{\tilde D}
|a(y,t/\a)|^{q}d\gamma(y)\frac{dt}{t}\Big)^\frac1q{\gamma(B(c,\a
r))}^{\frac{1}{q'}}
\\ & \qquad \lesssim \Big(\iint_{\tilde D}
|a(y,t/a)|^{q}d\gamma(y)\frac{dt}{t}\Big)^\frac1q{\gamma(B(c,r))}^{\frac{
1}{q'}}
\\ & \qquad =\Big(\iint_{D}
|a(y,t)|^{q}d\gamma(y)\frac{dt}{t}\Big)^\frac1q{\gamma(B(c,
r))}^{\frac{1}{q'}}
\\ & \qquad \leq 1.
\end{align*}
This shows that $ J\tilde a$ belongs to $L^1(\R^n,d\g(x); L^{q}(\tilde
D,d\gamma(y)\frac{dt}{t}))$ with norm $ \lesssim 1$.

An appeal to Theorem \ref{thm:atomic} now shows that $ J \tilde f\in
L^1(\R^n,d\g(x); L^{q}(\tilde D,d\gamma(y)\frac{dt}{t}))$ for all $f\in T^{1,q}(\g)$. 
The estimate \eqref{eq:tildeJ} then follows from the closed graph theorem. 
This completes the proof.
\end{proof}

\noindent
{\em Acknowledgment} -- We thank the anonymous referee for his/her very careful and interesting comments.

\end{document}